\documentclass[11pt, a4paper]{article}

\usepackage{a4wide}
\usepackage{amsthm,amsmath,amssymb,mathtools}
\usepackage{hyperref}

\theoremstyle{plain}
\newtheorem{theorem}{Theorem}
\newtheorem{claim}{Claim}
\newtheorem{conjecture}{Conjecture}
\newtheorem{lemma}[theorem]{Lemma}

\begin{document}

\title{Large induced subgraphs with $k$ vertices of almost maximum degree}
\author{Ant\'onio Gir\~ao \thanks{Department of Pure Mathematics and Mathematical Statistics, University of Cambridge, Cambridge, UK; \texttt{A.Girao@dpmms.cam.ac.uk}} 
        \and Kamil Popielarz \thanks{Department of Mathematics, University of Memphis, Memphis, Tennessee; \texttt{kamil.popielarz@gmail.com}}}
\maketitle

\abstract{
    In this note we prove that for every integer $k$, there exist constants $g_{1}(k)$ and $g_{2}(k)$ such that the following holds.
    If $G$ is a graph on $n$ vertices with maximum degree $\Delta$ then it contains an induced subgraph $H$ on at least $n - g_{1}(k)\sqrt{\Delta}$ vertices, such that $H$ has $k$ vertices of the same degree of order at least $\Delta(H)-g_{2}(k)$.
    This solves a conjecture of Caro and Yuster up to the constant $g_{2}(k)$.
}

\section{Introduction}
Given a graph $G$, let the repetition number, denoted by $rep(G)$, be the maximum multiplicity of a vertex degree. 
Trivially, any graph $G$ of order at least two contains at least two vertices of the same degree, i.e. $rep(G)\geq 2$.
This parameter has been widely studied by several researchers (e.g., \cite{Morayne, Bscott, WestCaro, ChenSchelp, Erdos}), in particular, by Bollob\' as and Scott, who showed that for every $k\geq 2$ there exist triangle-free graphs on $n$ vertices with $rep(G)\leq k$ for which $\alpha(G)= (1+o(1))n/k$ (\cite{Bscott}).
As there are infinitely many graphs having repetition number two, it is natural to ask what is the smallest number of vertices one needs to delete from a graph in order to increase the repetition number of the remaining induced subgraph.
This question was partially answered by Caro, Shapira and Yuster in \cite{YusterShapiraCaro}, indeed, they proved that for every $k$ there exists a constant $C(k)$ such that given any graph on $n$ vertices one needs to remove at most $C(k)$ vertices and thus obtain an induced subgraph with at least $\min\{k, n-C(k)\}$ vertices of the same degree. 
Related to this question, Caro and Yuster (\cite{CaroYuster}) considered the problem of finding the largest induced subgraph $H$ of a graph $G$ which contains at least $k$ vertices of degree $\Delta(H)$.
To do so they defined $f_k(G)$ to be the smallest number of vertices one needs to remove from a graph $G$ such that the remaining induced subgraph has its maximum degree attained by at least $k$ vertices. 
They found examples of graphs on $n$ vertices for which $f_2(G)\geq (1-o(1))\sqrt{n}$ and conjectured $f_k(G) \le O(\sqrt{n})$ for every graph $G$ on $n$ vertices. 
In the same paper they established the conjecture for $k \le 3$.

The following more general conjecture was posed recently by Caro, Lauri and Zarb in \cite{CaroMalta}.

\begin{conjecture}
    \label{conjecture_1}
 For every $k \geq 2$ there is a constant $g(k)$ such that given a graph $G$ with maximum degree $\Delta$, one can remove at most $g(k)\sqrt{\Delta}$ vertices such that the remaining subgraph $H\subseteq G$ has at least $k$ vertices of degree $\Delta(H)$. 
\end{conjecture}

Let us define $g(k,\Delta)=\max \{ f_k(G): \Delta(G)\leq \Delta\}$. 
In the same paper, they proved that $g(2,\Delta)=\lceil \frac{−3+ \sqrt{8\Delta+1}}{2} \rceil $ and stated that $g(3,\Delta)\leq 42\sqrt{\Delta}$. 
We should point out that, if true, the conjecture is best possible, as there are graphs on $n$ vertices found in \cite{CaroMalta} for which any induced subgraph on more than $n-\frac{k}{2}\sqrt{\Delta}$ does not contain $k$ vertices of the same maximum degree. 
We shall present such constructions in Section \ref{constructions}.

In this note we prove the following approximate version of Conjecture~\ref{conjecture_1}.
\begin{theorem}\label{theoremsamedegree}
    For every positive integer $k$, there exist constants $g_1(k)$ and $g_2(k)$ such that the following holds.
    If $G$ is a graph on $n$ vertices with maximum degree $\Delta$ then it contains an induced subgraph $H$ on at least $n - g_1(k)\sqrt{\Delta}$ vertices, such that $H$ has $k$ vertices of the same degree at least $\Delta(H) - g_2(k)$.
\end{theorem}

\section{Proofs}

Given a partition of $\{1,2,\ldots,n\}$ into $t$ sets, $A_1,A_2\ldots, A_t$, and a strictly decreasing sequence of non-negative integers  $r_1>r_2>r_3>\ldots>r_t$, we say $\mathcal{A}$ is an $(r_1,r_2,\ldots, r_t)$-uniform cover of $\{1,2\ldots, n\}$ if $\mathcal{A}$ is a multiset of subsets of $\{1,2,\ldots n\}$ such that, whenever $i \in \left\{ 1, \dots, t \right\}$ and $j \in A_{i}$, we have $|\left\{ A \in \mathcal{A}: j \in A \right\}| = r_{i}$.
Note that $\mathcal{A}$ is a multiset, hence we allow repetitions.

We call an $(r_1,r_2,\ldots,r_t)$-uniform cover $\mathcal{A}$ of $\{1,2,\ldots, n\}=A_1 \cup A_2\cup \ldots \cup A_t$ \textit{irreducible} if there is no proper $(r^{\prime}_1,\ldots,r^{\prime}_t)$-uniform cover $\mathcal{B} \subset \mathcal{A}$, for some strictly decreasing sequence of non-negative integers $r^{\prime}_1>r^{\prime}_2>\ldots>r^{\prime}_t$.

Given a uniform cover $\mathcal{A}$ of $\{1,2,\ldots,n\}$ and a subset $B\subseteq \{1,2,\ldots,n\}$ we define $w_{\mathcal{A}}(B)$ to be the number of times $B$ appears in $\mathcal{A}$.

\begin{lemma}
    \label{lemma_cover}
For all $n\in \mathbb{N}$, there exists $f(n)$ such that for any $1\leq t \leq n$ and any partition of $\{1,2,\ldots,n\}$ into $t$ sets $A_1,A_2,\ldots,A_t$, every $(r_1,r_2,\ldots,r_t)$-uniform cover $\mathcal{A}$ of $\{1,2,\ldots,n\}$ contains a proper $(r^{\prime}_1,r^{\prime}_2,\ldots,r^{\prime}_t)$-uniform sub-cover $\mathcal{B}\subset \mathcal{A}$ with $r^{\prime}_1\leq f(n)$.
\end{lemma}

\begin{proof} We shall prove there are only finitely many \textit{irreducible} covers. 
    For otherwise, let us assume there exists an infinite sequence  $\{B_i\}_{i\in \mathbb{N}}$ of \textit{irreducible} uniform covers.
    Since there are only finitely many partitions of $\{1,2,\ldots, n\}$, we may pass to an infinite subsequence $\{B_{l_i}\}_{i\in \mathbb{N}}$ of uniform covers of the same partition of $\{1,2,\ldots,n\}$. 
Now, choose $A \subseteq \{1,2,\ldots,n \}$ and consider the sequence of non-negative integers $\{w_{B_{l_i}}(A)\}_{i \in \mathbb{N}}$, clearly it must contain an infinite non-decreasing subsequence 
$w_{B_{l_{i_1}}}(A)\leq w_{B_{l_{i_2}}}(A)\leq \ldots $. We restrict our attention to this subsequence of the uniform covers $B_{l_{i_1}},B_{l_{i_2}},\ldots$ and iteratively apply the same argument for the remaining subsets of $\{1,2,\ldots,n\}$, always passing to a subsequence of the previous sequence of uniform covers. After we have done it for every subset of $\{1,2,\ldots,n\}$, we must end up with two  distinct \textit{irreducible} uniform covers (actually an infinite sequence)  $\mathcal{A}, \mathcal{B}$ for which $w_{\mathcal{A}}(F)\leq w_{\mathcal{B}}(F)$ for every $F\subseteq \{1,2,\ldots,n\}$. This implies $\mathcal{A}\subseteq \mathcal{B}$, which is a contradiction. Take $f(n)$ to be the maximum $r_1$ over all \textit{irreducible} uniform covers of $\{1,2,\ldots,n\}$. 
\end{proof}
\begin{lemma}
    \label{lemma_bipartite}
    For every $n \in \mathbb{N}$, there exists $f(n)$ such that the following holds.
    Let $G=(A,B)$ be a bipartite graph with $A=\{x_1,x_2,\ldots,x_n\}$.
    Then there exists a subset $W\subseteq V(B)$ of size at most $n \cdot f(n)=f^{\prime}(n)$, such that the induced bipartite graph ${G^{\prime}=G[A,(B\setminus W)]}$ has the property that
    \[\text{if } d_{G}(x_i)>d_{G}(x_j)  \text{ then } d_{G}(x_i)-d_{G^{\prime}}(x_i)>d_{G}(x_j)-d_{G^{\prime}}(x_j).
    \]
\end{lemma}
\begin{proof}
    Partition $A$ into $A_{1},\dots, A_{t}$, so that two vertices belong to the same part if they have the same degree.
    Let $r_{i}$ be the degree of the vertices in $A_{i}$.
    We may assume that $r_{1} > r_{2} > \dots >  r_{t}$.
   The lemma follows as a corollary of Lemma~\ref{lemma_cover}.
    Indeed, for every vertex $w\in B$, let $A_{w}\subseteq \{1,2,\ldots,n\}$ such that $i\in A_{w}$ if $x_i$ is a neighbour of $w$ in $G$. Note that $\mathcal{A}=\{A_{w}: w\in B \}$ is an $(r_{1},r_{2},\ldots, r_{t})$-uniform cover of $\{1,2,\ldots,n\}$. 
    Applying now Lemma~\ref{lemma_cover}, we can find a $(r^{\prime}_1,r^{\prime}_2,\ldots, r^{\prime}_t)$-uniform sub-cover $\mathcal{B}\subseteq \mathcal{A}$ with $r^{\prime}_{1}\leq f(n)$.
    Let $W=\{w\in B: A_w\in \mathcal{B}\}$ and $G^{\prime}=G[A,(B\setminus W)]$. It is easy to see that $|W|\leq n\cdot f(n)$ and that the property is satisfied by the definition of uniform cover. 
\end{proof}

Given a positive integer $k$ and a graph $G$ with the vertex set $\left\{ x_{1}, \dots, x_{n} \right\}$ such that $d(x_{1}) \ge \dots \ge d(x_{n})$, let $r_{k}(G) := \Delta(G) - d_{G}(x_{k})$ be the difference between the maximum degree and the degree of vertex $x_{k}$.

\begin{theorem}
    \label{theorem_1}
    For every positive integer $k$ there exists $h(k)$ such that the following holds.
    If $G$ is a graph on $n$ vertices with maximum degree $\Delta$ then it contains an induced subgraph $H$ on at least $n - (h(k)+k)\sqrt{\Delta}$ vertices, such that $r_{k}(H) \le h(k)\cdot k$.
\end{theorem}

\begin{proof}
    The proof consists of two parts.
    Firstly, we shall show that we can remove at most $k\sqrt{\Delta}$ vertices from $G$ so that in the remaining graph $H'$ we have $r_{k}(H') \le \sqrt{\Delta}$.
    Then we iteratively apply Lemma~\ref{lemma_bipartite} (at most $\sqrt{\Delta}$ times) in order to obtain an induced subgraph $H$ of $H'$ on at least $n-\left(h(k) + k\right)\sqrt{\Delta}$ vertices such that $r_{k}(H) \le h(k)\cdot k$.
    We may take $h(k)$ to be $f^{\prime}(k)$ from Lemma~\ref{lemma_bipartite}.

    We start by showing there is a large induced subgraph $H' \subseteq G$ with $r_{k}(H') \le \sqrt{\Delta}$.
    \begin{claim}
        \label{claim_1}
        There is an induced subgraph $H'$ of $G$ on at least $n - k\sqrt{\Delta}$ vertices such that $r_{k}(H') \le \sqrt{\Delta}$.
    \end{claim}
    \begin{proof}[Proof of Claim~\ref{claim_1}]
        Consider the following procedure.
        Let $G_{0} = G$ and suppose that ${G_{0} \supset \dots \supset G_{i}}$ have been defined.
        If $G_{i}$ does not have the required property then let $G_{i+1}$ be obtained from $G_{i}$ by removing $k$ vertices with largest degrees in $G_i$.
        Notice that $\Delta(G_{i+1}) \le \Delta(G_{i}) - \sqrt{\Delta}$ and $|G_{i+1}| = |G_{i}| - k$.
        Observe that the procedure will stop after at most $\sqrt{\Delta}$ steps, as otherwise the obtained graph would have maximum degree $0$.
        Since $|G_{i}| \ge n - i\cdot k$ we have that $|H'| \ge n - k\sqrt{\Delta}$.
    \end{proof}

    We now proceed to the second part of the proof and iteratively apply Lemma~\ref{lemma_bipartite}.
    In each step we remove at most $h(k)$ vertices from $H'$ while decreasing the value of $r_{k}$ and we stop when $r_{k}$ is at most $ k \cdot h(k)$.

    Let $H_{0} = H'$ and suppose that $H_{0}, \dots, H_{i}$ have already been defined.
    If $r_{k}(H_{i}) \le k \cdot h(k)$ then we are done, so we may assume that $r_{k}(H_{i}) > k\cdot h(k)$.
    Let $A = \left\{x_{1}, \dots, x_{k}  \right\}$ be a set of $k$ vertices with the largest degrees in $H_{i}$ and write $B$ for $H_{i} \setminus A$.
    Without loss of generality we may assume that $d_{H_{i}}(x_{1}) \ge \dots \ge d_{H_{i}}(x_{k})$.
    Since $r_k(H_i) \geq k \cdot h(k)$ there must exist $l \in \left\{ 2, \dots, k \right\}$ such that $d_{H_{i}}(x_{l}) >  d_{H_{i}}(x_{l-1}) + h(k)$.
    Now consider the bipartite subgraph $K = H_{i}[A, B]$. By Lemma~\ref{lemma_bipartite}, with $G = K$ and $n=k$, we can remove a set $W\subset B$ of at most $f^{\prime}(k)=h(k)$ vertices from $B$, and obtain $K' = H_i[A, (B \setminus W)]$ such that
    \begin{equation}
        \label{eq_1}
        \text{for any } x, y \in A\text{, if } d_{K}(x) < d_{K}(y)\text{ then }d_{K}(x) - d_{K'}(x) < d_{K}(y) - d_{K'}(y).
    \end{equation}
    Let $H_{i+1} = H_{i} \setminus W$ (hence $|H_{i+1}| \ge |H_{i}| - |W| \ge |H_{i}| - h(k)$).
    The following claim asserts that the above procedure will stop after at most $\sqrt{\Delta}$ steps.
    \begin{claim}
        \label{claim_2}
        $r_{k}(H_{i+1}) < r_{k}(H_{i})$.
    \end{claim}
    \begin{proof}[Proof of Claim~\ref{claim_2}]
        Let $z$ be a vertex with the maximum degree and $w$ a vertex with the $k$'th largest degree in $H_{i+1}$.
        Observe that $z = x_{t}$ for some $t \ge l$ and $d_{H_{i+1}}(w) \ge d_{H_{i+1}}(x_{s})$ for some $s < l$.
        First, notice that $d_{H_{i}}(x_{t}) - d_{H_{i}}(x_{s}) \le d_{H_{i}}(x_{1}) - d_{H_{i}}(x_{k}) = r_{k}(H_{i})$.
        Hence, $r_{k}(H_{i+1}) = d_{H_{i+1}}(z) - d_{H_{i+1}}(w) \le d_{H_{i+1}}(x_{t}) - d_{H_{i+1}}(x_{s}) < d_{H_{i}}(x_{t}) - d_{H_{i}}(x_{s}) \le r_{k}(H_{i})$,
        where the strict inequality follows from (\ref{eq_1}) since $d_{K}(x_{t}) > d_{K}(x_{s})$.
    \end{proof}

    As in each iteration the value of $r_{k}$ decreases, we must stop after at most $r_{k}(H') = \sqrt{\Delta}$ steps thus getting an induced subgraph $H \subset H'$ with $r_{k}(H) \le k\cdot h(k)$ and $|H| \ge |H'| - h(k)\sqrt{\Delta}\geq n-(h(k)+k)\sqrt{\Delta}$.
\end{proof}

In order to prove Theorem~\ref{theoremsamedegree} we need the following theorem of Caro, Shapira and Yuster, appearing in \cite{YusterShapiraCaro}, whose proof is inspired by the one used by Alon and Berman in \cite{AlonBerman}. 

\begin{theorem}
    \label{lemma_caro}
    For positive integers $r, d, q,$ the following holds.
    Any sequence of $n \ge \left( \lceil q/r\rceil + 2 \right)\left( 2rd + 1 \right)^{d}$ elements of $\left[ -r, r \right]^{d}$ whose sum, denoted by $z$, is in $\left[ -q, q \right]^{d}$ contains a subsequence of length at most $\left( \lceil q/r\rceil +2\right)\left( 2rd + 1 \right)^{d}$ whose sum is $z$.
\end{theorem}

As usual, we write $R(k)$ (see e.g. \cite{ModernBollobas}) for the \textit{two coloured Ramsey number}, the least integer $n$ such that in any two colouring of the edges of the complete graph on $n$ vertices, there is a monochromatic $K_{k}$. 

\begin{proof}[Proof of Theorem~\ref{theoremsamedegree}]
    Firstly, we apply Theorem~\ref{theorem_1} with $k = R(k)$ to find a large induced subgraph $G' \subset G$  of order at least $n^{\prime}\geq n-(h(R(k))+R(k)) \sqrt{\Delta}$ and with vertex set $\left\{ x_{1}, \dots, x_{n'} \right\}$ where $d(x_{1})\ge d(x_{2}) \ge \dots \ge d(x_{n'})$ and $d(x_{1}) - d(x_{R(k)}) \le {h(R(k))\cdot R(k) = M}$.
    Now we follow the proof of Theorem~1.1 in \cite{YusterShapiraCaro}.

    By the definition of $R(k)$ we can find a set $S$ of $k$ vertices in $\left\{ x_{1}, \dots, x_{R(k)} \right\}$ that induces either a complete graph or an independent set.

    Without loss of generality, assume that $S = \left\{ v_{n'-k+1}, \dots, v_{n'} \right\}$ and $V(G) \setminus S = \left\{ v_{1}, \dots, v_{n'-k} \right\}$.
    Let $e(v_{i}, v_{j})$ be equal to $1$ if there is an edge between $v_{i}$ and $v_{j}$, and $0$ otherwise.
    We construct a sequence $X$ of $n'-k$ vectors $w_{1}, \dots, w_{n'-k}$ in $\left[ -1, 1 \right]^{k-1}$ as follows.
    The coordinate $j$ of $w_{i}$ is $e(v_{n'-k+j}, v_{i}) - e(v_{n'}, v_{i})$ for $i = 1, \dots, n'-k$ and $ j = 1, \dots,k-1$.
    It is clear that $e(v_{n'-k+j}, v_{i}) - e(v_{n'}, v_{i}) \in [-1, 1]$ as required.
    Consider the sum of all the $j$'th coordinates,

    \begin{align*}
        \sum_{i=1}^{n'-k}\left( e(v_{n'-k+j}, v_{i}) - e(v_{n'}, v_{i}) \right) &= \sum_{i = 1}^{n'-k}e(v_{n'-k+j}, v_{i}) - \sum_{i=1}^{n'-k}e(v_{n'}, v_{i}) \\
        &= (d(v_{n'-k+j}) - a) - (d(v_{n'}) - a) = d(v_{n'-k+j}) - d(v_{n'})\\
        &\le M,
    \end{align*}
    where $a = k-1$ if $G'[S]$ is complete, and $a = 0$ otherwise.
    Hence,
    \[
        z = \sum_{i=1}^{n'-k}w_{i} \in \left[ -M,M  \right]^{k-1}.
    \]

    By Theorem~\ref{lemma_caro}, with $d = k-1$ and $q = M$, there is a subsequence of $X$ of size at most $(M+2)(2k-1)^{k-1}$ whose sum is $z$.
    Deleting the vertices of $G'$ corresponding to the elements of this subsequence results in an induced subgraph $H \subset G'$ in which all the $k$ vertices of $S$ have the same degree of order at least $\Delta(H) - \left( M + (M+2)(2k-1)^{k-1} \right)$. 
    Choosing $g_1(k)=g_2(k) =  h(R(k))(4k)^{k}$ we conclude the theorem.
\end{proof}

\section{Remarks}\label{constructions}

In the previous section, we proved that every graph contains a large induced subgraph with at least $k$ vertices having the same degree of order almost the maximum degree. Note that Theorem~\ref{theoremsamedegree} is sharp up to the size of the functions $g_1(k)$ and $g_2(k)$. 
Indeed, there are graphs for which one needs to remove ''roughly'' $\frac{k}{2}\sqrt{\Delta}$ vertices to force the remaining subgraph to have $k$ vertices with the same degree "near" the maximum degree. 
For any $k$ and $\Delta$, let $G^{\Delta}$ be the disjoint union of the stars $K_{1, n_{1}}, \dots, K_{1, n_{t}}$, where $n_{i} = i \cdot \sqrt{\Delta}$, for $i \in \left\{ 1, \dots, t = \sqrt{\Delta} \right\}$ and let $G_k^{\Delta}$ be the disjoint union of $k/2$ copies of $G^{\Delta}$.
It is easy to see that, for any constant $D$, one needs to remove at least $\frac{k}{2} \sqrt{\Delta} - \frac{k}{2}D$ vertices from $G_{k}^{\Delta}$ in order to obtain an induced graph $H$ with $k$ vertices of the same degree of order at least $\Delta(H) - D$. 

Whether removing $C(k)\sqrt{\Delta}$ vertices is enough to force the remaining induced subgraph to have at least $k$ vertices attaining the maximum degree remains an interesting open question. 

\section{Acknowledgments}

This project was carried through during our stay in IMT School for Advanced Studies Lucca.
We would like to thank the Institute for their kind hospitality.

\bibliographystyle{acm}
\bibliography{krepeatmax}

\end{document}